\documentclass[11pt]{article}
\usepackage[latin1]{inputenc} 
\usepackage{amsmath}
\usepackage{amsthm}
\usepackage{amscd}
\usepackage{amssymb}
\newtheorem{theorem}{Theorem}
\newtheorem{proposition}[theorem]{Proposition}

\newcommand{\smashprod}{\wedge}

\newcommand{\Z}{\mathbf{Z}}

\newcommand{\C}{\mathbf{C}}
\newcommand{\CP}{\C P}

\newcommand{\isom}{\approx}

\newcommand{\tensor}{\otimes}

\providecommand{\ker}{}
\renewcommand{\ker}{\mathrm{Ker}\,}

\begin{document} 

\title{The String Topology Loop Coproduct and Cohomology Operations}

\author{Anssi Lahtinen\thanks{The author would like to thank 
the Finnish Cultural Foundation for its support during the time
this note was finished.}}

\date{December 2, 2007 }

\maketitle

The purpose of this note is to explore the relationship between
cohomology operations in a generalized cohomology theory $h^*$ with
products and a string topology loop coproduct
\[
    c : h^*(LM) \to h^*(LM \times LM)
\]
dual to the Chas--Sullivan loop product \cite{CS}. The exact definition 
of $c$ will be given below. In more detail, we are 
interested in giving a description for the failure of commutativity
in the diagram
\begin{equation}
\label{cdiag}
\begin{CD}
  h^*(LM) @>c>> h^*(LM \times LM) \\
  @V{\alpha}VV    @VV{\alpha}V \\
  h^*(LM) @>c>> h^*(LM \times LM)  
\end{CD}
\end{equation}
when $\alpha$ is a cohomology operation of $h^*$, and will obtain a
satisfactory result, phrased in terms of an exotic module structure on
$h^*(LM)$ defined in terms of a characteristic class arising from
$\alpha$, in the case where the operation $\alpha$ preserves addition
and multiplication. Important examples of such operations include the
total Steenrod square
\[
  Sq = 1 + Sq^1 + Sq^2 + \cdots
\]
in ordinary $\Z_2$-cohomology and the Adams operations 
\[
  \psi^k : K(X) \to K(X) 
\]
in $K$-theory. In the case of the the total Steenrod square, our
result parallels an earlier one by Gruher and Salvatore, who in
\cite{GS} described the interaction of the Steenrod squaring
operations with another string topology product in terms of
Stiefel--Whitney classes.

The author expects the results in this note to appear as part of his
PhD thesis, to be written under the direction of Professor R.~L.~Cohen
at Stanford University.  In future work, the author plans to explore
the relationship between cohomology operations and other products
related to string topology, such as the fusion product in twisted
K-theory \cite{FHT3}.

We will now define the the loop coproduct. Suppose $M^d$ is a closed
oriented manifold, let $LM$ be the space of smooth loops in $M$, and
let $LM^{TM}$ denote the Thom space of the bundle $ev^*TM$, where $TM$
stands for the tangent bundle of M and
\[
   ev : LM \to M
\]
is evaluation at $1 \in S^1$. Then according to Cohen and 
Jones~\cite{htopyreal}, the Chas--Sullivan loop product 
\[
     H_*(LM) \tensor H_*(LM) \to H_{*-d}(LM)
\]
arises from a certain map of spaces
\[
     \mu : LM \times LM \to LM^{TM}
\]
and the Thom isomorphism 
\[
     \tilde{H}_*(LM^{TM}) \xrightarrow{\isom} H_{*-d}(LM).
\]
With the above interpretation of the loop product as motivation, we
define the \emph{loop coproduct} of $M$ in a generalized cohomology
theory $h^*$ with products to be the composition
\[
    c : h^*(LM) \xrightarrow[\isom]{Thom} \tilde{h}^*(LM^{TM}) 
            \xrightarrow{\mu^*} h^*(LM \times LM).
\]
Of course, to ensure the existence of the required Thom isomorphism,
we need to assume that our manifold $M$ has been given an orientation
with respect to $h^*$, and we will henceforth do so.  A natural
alternative way to define the loop coproduct would be to make use of
the ring spectrum structure of $LM^{-TM}$ described by Cohen and Jones
and take the loop coproduct to be the composition
  \begin{align*}
      h^*(LM)  \xrightarrow[\isom]{Thom} \tilde{h}^*(LM^{-TM})
              \to \tilde{h}^*(LM^{-TM} \smashprod LM^{-TM}) \\ 
               = {h}^*(LM \times LM ^ {-TM \times -TM} )
              \xrightarrow[\isom]{Thom^{-1}} h^*(LM \times LM).
  \end{align*}
It is easily checked that the two definitions agree.

Suppose now $\alpha$ is a cohomology operation of $h^*$ preserving
sums and products, so that
\[
     \alpha(x+y) = \alpha(x) + \alpha(y)
\]
and 
\[
     \alpha(xy) = \alpha(x)\alpha(y)
\]
for all $x,y \in h^*(X)$ and for any space $X$. We would like to
describe the failure of the diagram \eqref{cdiag} to commute.  Since
the operation $\alpha$ preserves sums, it extends to give an operation
(also denoted by $\alpha$) on the reduced cohomology groups
\[ 
   \tilde{h}^*(X) = \ker (h^*(X) \to h^*(pt)) 
\]
as well as the unreduced ones.  By naturality of $\alpha$ with respect
to maps of spaces, the right hand square in
\begin{equation}
\label{cdiag2}
\begin{CD}
     h^*(LM) @>{Thom}>{\isom}> \tilde{h}^*(LM^{TM}) @>{\mu^*}>> h^*(LM \times LM) \\
    @V{\alpha}VV                       @VV{\alpha}V                 @VV{\alpha}V  \\
     h^*(LM) @>{Thom}>{\isom}> \tilde{h}^*(LM^{TM}) @>{\mu^*}>> h^*(LM \times LM) 
\end{CD}  
\end{equation}
commutes, whence we see that any non-commutativity in the diagram
\eqref{cdiag} arises from the non-commutativity of the left hand
square in \eqref{cdiag2}.

The above observations imply that we should study the failure of
$\alpha$ to commute with the Thom isomorphism.  To do this in the
proper generality, let $X$ be a space, and let $\xi$ be a vector
bundle over $X$ equipped with a Thom class $u_{\xi} \in
\tilde{h}^*(X^{\xi}).$ Then by the Thom isomorphism theorem, we have
\[
    \alpha (u_{\xi}) = \rho_{\alpha}(\xi) \cdot u_{\xi} 
                     \in \tilde{h}^*(X^{\xi})
\]
for some unique class $\rho_{\alpha}(\xi) \in h^*(X)$.  From the
assumption that $\alpha$ preserves products, it now follows easily
that for any $a \in h^*(X)$ we have
\begin{equation}
\label{eqm}
    \alpha (a \cdot u_{\xi}) = \alpha (a) \cdot \alpha (u_{\xi}) 
                       = \alpha (a) \cdot (\rho_{\alpha}(\xi) \cdot u_{\xi})
		       = (\alpha (a)  \rho_{\alpha}(\xi)) \cdot u_{\xi},  
\end{equation}
whence we see that in a sense the class $\rho_{\alpha}(\xi)$
completely describes the failure of the square
\[
\begin{CD}
    h^*(X) @>{Thom}>{\isom}> \tilde{h}^*(X^{\xi}) \\
    @V{\alpha}VV                 @VV{\alpha}V \\
    h^*(X) @>{Thom}>{\isom}> \tilde{h}^*(X^{\xi})  
\end{CD}
\]
to commute.  We find it convenient to express this observation in the
following form.
\begin{proposition}
\label{thomprop}
  Let $\Z[\underline{\alpha}]$ be a polynomial ring over a single variable
$\underline{\alpha}$, and let $\underline{\alpha}$ act on $h^*(X)$ by
\[
  \underline{\alpha} \cdot x = \alpha(x) \rho_{\alpha}(\xi) \quad 
            \text{ for $x \in h^*(X)$}
\]
and on $\tilde{h}^*(X^\xi)$ by 
\begin{equation}
  \label{bmodstr}
  \underline{\alpha}\cdot x = \alpha(x) \quad 
            \text{ for $x \in \tilde{h}^*(X^\xi)$}.
\end{equation}
Then the Thom isomorphism map
\begin{equation*}
    h^*(X) \xrightarrow{\,\cdot u_{\xi}\,} \tilde{h}^*(X^{\xi})\\
\end{equation*}
becomes an isomorphism of $\Z[\underline{\alpha}]$-modules.
\end{proposition}
\begin{proof}
The statement is essentially just a reformulation of
\eqref{eqm}. Notice that the assumption that $\alpha$ preserves sums
is needed to guarantee that the given definitions make $h^*(X)$ and
$\tilde{h}(X^{\xi})$ into $\Z[\underline{\alpha}]$-modules.
\end{proof}

Since the $\Z[\underline{\alpha}]$-module structure on $h^*(Y)$
obtained by letting $\underline{\alpha}$ act as $\alpha$ as in
\eqref{bmodstr} is clearly natural with respect to maps induced by
maps of spaces, we obtain the following.
\begin{theorem}
\label{lcpthm}
Let $\underline{\alpha}$ act on $h^*(LM)$ by
\[
  \underline{\alpha}\cdot x = \alpha(x) \rho_{\alpha}(ev^* TM) \quad 
            \text{ for $x \in {h}^*(LM)$}
\]
and on $h^*(LM \times LM)$ by 
\[
  \underline{\alpha}\cdot x = \alpha(x) \quad 
            \text{ for $x \in {h}^*(LM \times LM)$}.  
\]
Then the loop coproduct 
\[
     c : h^*(LM) \to h^*(LM \times LM)
\]
is a map of $\Z[\underline{\alpha}]$-modules. \qed
\end{theorem}

In view of Proposition \ref{thomprop} and Theorem \ref{lcpthm}, we
should try to understand the classes $\rho_{\alpha}(\xi)$. The
following proposition summarizes their basic properties.
\begin{proposition}
\label{rhoprop}
The map associating to an $h^*$-oriented vector bundle $\xi$ over a
space $X$ the class $\rho_{\alpha}(\xi) \in h^*(X)$ has the following
properties:
\begin{enumerate}
  \item \label{it1}
        The class $\rho_{\alpha}(\xi) \in h^*(X)$ depends only on 
        the isomorphism class of 
        $\xi$ as an $h^*$\nobreakdash-oriented vector bundle over $X$;
  \item \label{it2} 
        $\rho_{\alpha} (f^*\xi)  = f^* \rho_{\alpha}(\xi) \in h^*(Y)$ when 
        $f$ is a map $X \to Y$;
  \item \label{it3} 
        Given  $h^*$-oriented vector bundles $\xi$ over $X$ and
        $\zeta$  over $Y$ with homogeneous Thom classes $u_\xi$ and
        $u_\zeta$, we have
        \[
             \big[ \rho_\alpha(\xi \times \zeta) \big]_k 
                 = \sum_{i+j=k} (-1)^{j|u_\xi|} \big[\rho_{\alpha}(\xi)\big]_i 
                                      \times \big[\rho_{\alpha}(\zeta)\big]_j
                  \in h^k(X\times Y)
        \]
        where $[ - ]_k$ denotes the degree $k$ part. In particular,
        \[
           \rho_{\alpha}(\xi \times \zeta) 
                =  \rho_{\alpha}(\xi) \times \rho_{\alpha}(\zeta) 
               \in h^*(X\times Y)
        \]
        if $h^*(X\times Y)$ consists of elements of order 2,
        if the degree of $u_\xi$ is even or if $\rho_{\alpha}(\zeta)$ 
        is concentrated in even degrees.
\end{enumerate}
\end{proposition}
\begin{proof}
Parts \ref{it1} and  \ref{it2} are trivial, and part \ref{it3} follows
from the computation
\begin{align*}
    \big[\rho_\alpha(\xi \times \zeta)\big]_k \cdot (u_\xi \smashprod u_\zeta) 
 &= \big[\alpha(u_{\xi} \smashprod u_{\zeta})  \big]_{k + |u_\xi|+ |u_\zeta|}\\
 &= \big[ \alpha(u_{\xi}) \smashprod \alpha(u_{\zeta}) \big]_{k 
           + |u_\xi|+ |u_\zeta|}\\
 &= \sum_{i+j=k}  \big( \big[ \rho_\alpha(\xi) \big]_i \cdot u_\xi \big)
    \smashprod  \big( \big[ \rho_\alpha(\zeta)\big]_j \cdot u_\zeta \big) \\
 &= \sum_{i+j=k} (-1)^{j|u_\xi|} \big( \big[ \rho_\alpha(\xi) \big]_i 
          \times \big[ \rho_\alpha(\zeta)\big]_j\big)
           \cdot (u_\xi \smashprod u_\zeta). \qedhere
\end{align*}
\end{proof}
Notice that parts \ref{it1} and \ref{it2} of the preceding proposition
together state that $\rho_{\alpha}$ is a characteristic class of 
$h^*$-oriented vector bundles, and that \ref{it2} and \ref{it3} combine 
to prove the formula
\begin{equation}
  \label{sumformula}
  \big[\rho_{\alpha}(\xi \oplus \zeta)\big]_k = 
    \sum_{i+j=k} (-1)^{j|u_\xi|}\big[\rho_{\alpha}(\xi)\big]_i 
                 \big[\rho_{\alpha}(\zeta)\big]_j
           \in h^k(X)
\end{equation}
when $\xi$ and $\zeta$ are $h^*$-oriented vector bundles over $X$
(with homogeneous Thom classes).  Also observe that in the case where
$h^*$ is ordinary cohomology with $\Z_2$-coefficients and $\alpha$ is
the total Steenrod square
\[
   Sq = 1 + Sq^1 + Sq^2 + \cdots,
\]
the class $\rho_{\alpha}(\xi)$ is simply the total Stiefel--Whitney
class $w(\xi) \in H^*(X;\,\Z_2)$. In the case where $h^*$ is complex
$K$-theory and $\alpha$ is the $k$-th Adams operation $\psi^k$, the
class $\rho_{\alpha}(\xi)$ is the cannibalistic characteristic class
$\rho_k(\xi) \in K(X)$ considered by Adams~\cite{JX2} (whence our
notation $\rho_{\alpha}$).  The following well-known result aids the
computation of the classes $\rho_{\alpha}$ in this case; together with
the splitting principle of complex vector bundles and the sum formula
\eqref{sumformula}, it in principle completely determines the classes
$\rho_k(\xi) \in K(X)$ for complex vector bundles $\xi$.
\begin{proposition}
Suppose $\lambda$ is a complex line bundle over a space $X$. Then 
\[ 
    \rho_k(\lambda) = 1 + \lambda + \cdots + \lambda^{k-1} \in K(X).
\]
\end{proposition}
\begin{proof}
It is enough to consider the case where $\lambda$ is the universal
line bundle $\eta$ over $\CP^\infty$. However, in this case the Thom
space $(\CP^\infty)^{\eta}$ is homeomorphic to $\CP^\infty$, with the
Thom class corresponding to the class $\eta - 1 \in
\tilde{K}(\CP^\infty)$ and the module structure of
$\tilde{K}((\CP^\infty)^{\eta})$ over $K(\CP^\infty)$ corresponding to
the usual module structure of $\tilde{K}(\CP^\infty)$ over
$K(\CP^\infty)$.  Now the claim follows from the computation
\[
    \psi^k(\eta - 1 ) = \eta^k - 1 = (1+ \eta+ \cdots + \eta^{k-1})(\eta - 1).
 \qedhere
\]
\end{proof}

\bibliography{loop-coproduct} 
\bibliographystyle{plain}

\end{document}